\documentclass[11pt,a4paper]{article}

\usepackage{caption}
\usepackage{amsmath,amsfonts,amsthm,amssymb}
\usepackage{graphicx,fancyhdr}
\usepackage{multirow}
\usepackage{dsfont}
\usepackage{subfigure}
\usepackage{bbm}
\usepackage{booktabs}
\DeclareMathAlphabet{\mathpzc}{OT1}{pzc}{m}{it}
\usepackage{yfonts}
\usepackage{natbib}
\usepackage{mathtools}
\usepackage{mathrsfs,float}
\usepackage{hyperref}
\usepackage{authblk}
\makeatletter
\usepackage[letterpaper,top=2.5cm,bottom=2.5cm,left=3.2cm,right=3.2cm,marginparwidth=1.75cm]{geometry}
\usepackage{bm}

\numberwithin{equation}{section}
\newtheorem{thm}{Theorem}[section]

\newtheorem{prob}[thm]{Problem}

\newtheorem{lem}[thm]{Lemma}

\title{Inverse source problem with a posteriori interior measurements for space-time fractional diffusion equations}
\author[a]{Kai Yu}
\author[a]{Zhiyuan Li}
\author[b]{Yikan Liu\thanks{Corresponding author. E-mail: liu.yikan.8z@kyoto-u.ac.jp}}
\affil[a]{School of Mathematics and Statistics, Ningbo University, Ningbo 315211, China}
\affil[b]{Department of Mathematics, Kyoto University, Kitashirakawa-Oiwakecho, Sakyo-ku, Kyoto 606-8502, Japan}

\date{}

\begin{document}

\maketitle

\begin{abstract}
This paper investigates an inverse source problem for space-time fractional diffusion equations from a posteriori interior measurements. The uniqueness result is established by the memory effect of fractional derivatives and the unique continuation property. For the numerical reconstruction, the inverse problem is reformulated as an optimization problem with the Tikhonov regularization. We use the Levenberg-Marquardt method to identify the unknown source from noisy measurements. Finally, we give some numerical examples to illustrate the efficiency and accuracy of the proposed algorithm.
\end{abstract}



\section{Introduction and the main result}

Let $\Omega$ be a bounded domain in $\mathbb R^d$ with a sufficiently smooth boundary $\partial\Omega$ and let $T>0$ be a fixed final time. In this paper, we consider the following initial-boundary value problem for a space-time fractional diffusion equation
\begin{equation}\label{eq-gov}
\begin{cases}
\partial_t^\alpha(u(x,t)-\phi(x))+\mathcal A^\beta u(x,t)=f(x)\mu(t), & (x,t)\in\Omega\times(0,T),\\
u(x,\,\cdot\,)-\phi(x)\in H_\alpha(0,T), & \mbox{a.e.}\ x\in\Omega,\\
u(x,t)=0, & (x,t)\in\partial\Omega\times(0,T).
\end{cases}
\end{equation}
Here by $\partial_t^\alpha$ and $\mathcal A^\beta$ we denote the fractional derivative operator of order $\alpha\in(0,1)$ in time and the fractional elliptic operator of order $\beta\in(0,1)$ in space respectively, which will be defined precisely in the next section. 

Along with the intensive attention paid to forward problems for space-time fractional diffusion equations, there has been a rapid growth in research on the corresponding inverse problems. For example, Zhang et al.\! \cite{2018Bayesian} employed a Bayesian approach to address the inverse problem for a time-space fractional diffusion equation, but they did not give a uniqueness result. Li and Wei \cite{2018An} investigated the identification of a time-dependent source component using initial and boundary data as well as extra measurement data at an inner location in a space-time fractional diffusion problem. Janno \cite{2020Determination} investigated an inverse problem for identifying time-dependent sources and parameters in nonlocal  diffusion and wave equations from final data. Li et al.\! \cite{2020Tikhonov} employed a Tikhonov regularization approach to solve the inverse problem of identifying a space-dependent source term in a time-space fractional diffusion equation. Tatar et al.\! in \cite{2015Determination,2016Simultaneous,2015An} solved an inverse space-dependent source problem and identified the orders of time and space fractional derivatives for a time-space fractional diffusion equation. Jia et al. in \cite{2017Harnack's} investigated Harnack's inequality for a space-time fractional diffusion equation and applications to an inverse source problem. In \cite{2021A fractional}, Djennadi et al. studied backward and inverse source problems in the time-space-fractional diffusion equations by fractional Tikhonov regularization method. Nie et al. in \cite{2023An inverse} investigated an inverse random source problem for the time-space fractional diffusion equation driven by fractional Brownian motion. Van et al. used the quasi-reversibility method to study an inverse source problem for time-space fractional parabolic equations in \cite{2023The}.

The aforementioned studies mostly based on global observation data at the terminal time in the whole domain $\Omega$ or in a subdomain over the whole time interval $(0,T)$. However, in many applications, obtaining such global data often encounters practical difficulties. For instance, acquiring global observation data is sometimes highly expensive or unrealistic due to certain restrictions on observation conditions. Especially for the observation in the time direction, it is always difficult to predict unexpected accidents in advance and carry out the observation throughout $(0,T)$. Alternatively, in practice it seems more preferable to endow some flexibility in choosing the observation interval in time.  

Bearing the above motivation in mind, we focus on the uniqueness issue of the following inverse source problem concerning \eqref{eq-gov} with a posteriori interior observation data.

\begin{prob}\label{prob-isp}
Let $T_1<T$ be a fixed positive constant and $\omega$ be a nonempty open subdomain of $\Omega$. Under certain assumptions, can we uniquely determine the source term $f(x)$ in $\Omega$ by measuring $u$ in $\omega\times(T_1,T)$?
\end{prob}

Several recent studies e.g., \cite{2019An,2023Uniqueness,2023Inverse,2022Simultaneous} asserted the possibility of utilizing the memory effect (see the definition in \eqref{def-caputo}) of fractional derivatives to reconstruct unknown quantities from observations on the time interval away from the initial time. However, to the best of our knowledge, there seems no literature incorporating also the nonlocality in space. In this article, combining the unique continuation in Lemma \ref{lem-ucp} and Duhamel's principle in Lemma \ref{lem-dp}, we establish the following main result on the uniqueness of Problem \ref{prob-isp}.

\begin{thm}\label{thm-isp}
Let $\phi=0,$ $f\in L^2(\Omega),$ $\mu\in L^2(0,T)$ and $u$ be the solution to \eqref{eq-gov}. Assume $f=0$ in a nonempty open subdomain $\omega\subset\Omega$. If $\mu$ satisfies
\begin{equation}\label{eq1.2}
\exists\,T_0>0\mbox { such that }\mu\not\equiv0\mbox{ in }(0,T_0),
\end{equation}
then for any $T\ge T_0$ and any $T_1\in(0,T),$ $u=0$ in $\omega\times(T_1,T)$ implies $f=0$ in $\Omega$.
\end{thm}

In the context of environmental accidents, the key assumption \eqref{eq1.2} means that the known temporal component $\mu$ of the contaminant is active only within $(0,T_0)$. Then Theorem \ref{thm-isp} requires terminating the observation after $\mu$ becomes inactive, whereas the starting time can be rather arbitrary. We call it a posteriori data to emphasize that the observation is not necessarily started from $t=0$. Meanwhile, the assumption $f=0$ in $\omega$ means that the value of $f$ is known in the observation area, which is also reasonable.

The remainder of this paper is organized as follows. In Section \ref{sec-pre}, we prepare necessary ingredients for the proof of Theorem \ref{thm-isp}, which is provided in Section \ref{sec3}. We propose the Levenberg-Marquardt method to solve the inverse problem in Section \ref{sec4}. Then Section \ref{sec5} is devoted to the numerical treatment of Problem \ref{prob-isp} with examples. Finally, Section \ref{sec6} closes this article with concluding remarks.

\section{Preliminaries}\label{sec-pre}

In this section, we first fix notations and introduce tools for later use, and then mention the well-posedness of the forward problem.  Let $L^2(\Omega)$ be a usual $L^2$-space with the inner product $(\,\cdot\,,\,\cdot\,)$, and $H_0^1(\Omega),H^2(\Omega)$ etc.\! denote the usual Sobolev spaces. By $H^\alpha(0,T)$, we denote the Sobolev-Slobodeckij space (e.g., Adams \cite{AF03}).

The fractional derivative $\partial_t^\alpha:H_\alpha(0,T)\longrightarrow L^2(0,T)$ of order $\alpha\in(0,1)$ is defined as the inverse of the $\alpha$-th order Riemann-Liouville integral operator of order $\alpha$:
\begin{equation}\label{def-caputo}
I^\alpha:L^2(0,T)\longrightarrow L^2(0,T),\quad I^\alpha g(t):=\int_0^t\frac{\tau^{\alpha-1}}{\Gamma(\alpha)}g(t-\tau)\,\mathrm d\tau,
\end{equation}
where $\Gamma(\,\cdot\,)$ denotes the Gamma function. According to Kubica, Ryszewska and Yamamoto \cite{KR20}, the domain $H_\alpha(0,T)$ of $\partial_t^\alpha$ is a subspace of $H^\alpha(0,T)$ defined by
\[
H_\alpha(0,T)=\left\{\begin{alignedat}{2}
& \left\{\psi\in H^\alpha(0,T)\mid\psi(0)=0\right\}, & \quad & 1/2<\alpha<1,\\
& \left\{\psi\in H^{1/2}(0,T)\mid\int_0^T\frac{|\psi(t)|^2}t\mathrm d t<\infty\right\}, & \quad & \alpha=1/2,\\
& H^\alpha(0,T), & \quad & 0<\alpha<1/2,
\end{alignedat}
\right.
\]
which is equipped with the norm
\[
\|\psi\|_{H_\alpha(0,T)}=\left\{\begin{alignedat}{2}
& \|\psi\|_{H^\alpha(0,T)}, & \quad & 0<\alpha<1,\ \alpha\neq1/2,\\
& \left(\|\psi\|_{H^{1/2}(0,T)}^2+\int_0^T\frac{|\psi(t)|^2}t\mathrm d t\right)^{1/2}, & \quad &\alpha=1/2.
\end{alignedat}\right.
\]
For further details, see e.g.\! Gorenflo, Luchko and Yamamoto \cite{GL15}. Here, it should be mentioned that $u(x,\,\cdot\,)-\phi(x)\in H_\alpha(0,T)$ for a.e.\! $x\in\Omega$ generalizes the common initial condition $u=\phi$ in $\Omega\times\{0\}$, and it only makes the usual sense for $\alpha>1/2$ (see e.g.\! \cite{GL15,KR20}). Moreover, it is not difficult to see that $\partial_t^\alpha$ is equal to the Riemann-Liouville derivative $D_t^\alpha:=\frac{\mathrm d}{\mathrm d t}\circ I^{1-\alpha}$ in $H_\alpha(0,T)$, where $\circ$ denotes the composite. 

We recall the Mittag-Leffler function
\[
E_{\alpha,\beta}(z):=\sum_{k=0}^\infty\frac{z^k}{\Gamma(\alpha k+\beta)},\quad z\in \mathbb C,\ \alpha>0,\ \beta\in\mathbb R
\]
along with the useful formula of its Laplace transform
\[
\mathcal{L}\left\{t^{\alpha-1} E_{\alpha,\beta}\left( \pm \lambda t^{\alpha}\right)\right\}(z)=\frac{z^{\alpha-\beta}}{z^{\alpha} \mp \lambda},~~ \Re(z)>|\lambda|^{\frac{1}{\alpha}},~~\alpha,\beta >0.
\]
The following properties are well known and can be found e.g.\! in Podlubny \cite{1998Fractional}

\begin{lem}\label{lem2.1}
Let $0<\alpha<2,$ $\beta\in\mathbb R$ be arbitrary and $\mu$ satisfy $\pi\alpha/2<\mu<\min\{\pi,\pi\alpha\}$. Then there exists a constant $c=c(\alpha,\beta,\mu)>0$ such that
\[
|E_{\alpha,\beta}(z)|\leq\frac c{1+|z|},\quad\mu\leq|\arg (z)|\leq\pi.
\]
\end{lem}

\begin{lem}\label{lem2.13}
Let $\alpha>0$ and $\lambda>0$. Then there hold $E_{\alpha,\alpha}(-t)>0,$ $E_{\alpha,1}(-t)>0$ for any $t\ge0$ and 
\[
\frac{\mathrm d}{\mathrm d t}E_{\alpha,1}(-\lambda t^\alpha)=-\lambda t^{\alpha-1} E_{\alpha,\alpha}(-\lambda t^\alpha),\quad t>0.
\]
\end{lem}

Next, the elliptic operator $\mathcal A$ in \eqref{eq-gov} is defined for $\varphi\in D(\mathcal A):=H^2(\Omega) \cap H_0^1(\Omega)$ as
\[
\mathcal A\varphi(x):=-\sum_{i,j=1}^d\partial_j(a_{i j}(x)\partial_i\varphi(x))+c(x)\varphi(x),\quad x\in\Omega,
\]
where $a_{i j}=a_{j i}\in C^1(\overline\Omega)$ ($1\le i,j\le d$) and $0\le c\in C(\overline\Omega)$. Moreover, there exists a constant $\delta>0$ such that
\[
\delta\sum_{i=1}^d\xi_i^2\le\sum_{i,j=1}^d a_{i j}(x)\xi_i\xi_j,\quad\forall\,x\in\overline\Omega,\ \forall\,(\xi_1,\dots,\xi_d)\in\mathbb R^d.
\]
For the definition of the fractional elliptic operator, we first introduce the eigensystem $\{(\lambda_n,\varphi_n)\}_{n=1}^\infty$ of $\mathcal A$, namely $\mathcal A\varphi_n=\lambda_n\varphi_n$ with $\varphi_n\in H^2(\Omega)\cap H_0^1(\Omega)$. Then it is known that $0<\lambda_1\le\lambda_2\le\cdots$, $\lambda_n\rightarrow\infty$ as $n\to\infty$, and $\{\varphi_n\}_{n=1}^{\infty}$ form a complete orthonormal basis of $L^2(\Omega)$. Now we can define the fractional Sobolev space $D(\mathcal A^\beta)$ and the corresponding fractional elliptic operator $\mathcal A^\beta$ as
\begin{align*}
D(\mathcal A^\beta) & :=\left\{u\in L^2(\Omega)\mid\|u\|_{D(\mathcal A^\beta)}^2:=\sum_{n=1}^\infty\left|\lambda_n^\beta(u,\varphi_n)\right|^2<\infty\right\},\\
\mathcal A^\beta u & :=\sum_{n=1}^\infty\lambda_n^\beta(u,\varphi_n)\varphi_n.
\end{align*}
 Henceforth we set $D(\mathcal A^{-\beta}) = (D(\mathcal A^\beta))'$, which is the vector space consisting of all bounded linear functionals on $D(\mathcal A^\beta)$. We denote the parting of $\psi\in D(\mathcal A^{-\beta})$ and $u\in D(\mathcal A^\beta)$ by ${ }_{-\beta}\langle\psi, u\rangle_\beta$. We note that $D(\mathcal A^{-\beta})$ is a Hilbert space with the norm
\[
\|u\|_{D(\mathcal A^{-\beta})}=\left(\sum_{n=1}^\infty\lambda_n^{-2\beta}|{ }_{-\beta}\langle u,\varphi_n\rangle_\beta|^2\right)^{1/2}.
\]

Based on the above settings, we state the well-posedness of the forward problem \eqref{eq-gov}.

\begin{thm}\label{thm2.2}
If $\phi\in L^2(\Omega),$ $f\in L^2(\Omega)$ and $\mu\in L^2(0,T),$ then there exists a unique solution to \eqref{eq-gov} such that $u\in L^2(0,T;D(\mathcal A^{\beta/2}))$ and $u-\phi\in H_\alpha(0,T;D(\mathcal A^{-\beta/2}))$. Moreover, the solution allows the explicit representation
\begin{align}
u(x, t) & =\sum_{n=1}^\infty E_{\alpha,1}(-\lambda_n^\beta t^\alpha)\phi_n\varphi_n(x)\nonumber\\
& \quad\,+\sum_{n=1}^\infty\int_0^t\mu(\tau)(t-\tau)^{\alpha-1}E_{\alpha,\alpha}(-\lambda_n^\beta(t-\tau)^\alpha)\,\mathrm d\tau f_n\varphi_n(x),\label{eq3.10}
\end{align}
where $\phi_n=\left(\phi,\varphi_n\right)$ and $f_n=(f,\varphi_n)$. Further, we have the estimate
\[
\|u\|_{L^2(0,T;D(\mathcal A^{\beta/2}))}+\|u-\phi \|_{H_\alpha(0,T;D(\mathcal A^{-\beta/2}))}\leq C\left(\|\phi\|_{L^2(\Omega)}+\|\mu\|_{L^2(0,T)} \|f\|_{L^2(\Omega)}\right),
\]
where $C>0$ is a constant depending on $\alpha,T,\Omega$.
\end{thm}

\begin{proof}
Using the traditional argument of eigenfunction expansion, we can obtain a weak solution for the direct problem \eqref{eq-gov} as \eqref{eq3.10}. We divide the proof process into the two steps.\medskip

{\bf Step 1 } We prove $u\in  L^2(0,T;D(\mathcal A^{\beta/2}))$. Setting
\begin{align*}
u_1(x,t) & :=\sum_{n=1}^\infty E_{\alpha,1}(-\lambda_n^\beta t^\alpha)\phi_n\varphi_n(x),\\
u_2(x,t) & :=\sum_{n=1}^\infty\int_0^t\mu(\tau)(t-\tau)^{\alpha-1}E_{\alpha,\alpha}(-\lambda_n^\beta(t-\tau)^\alpha)\,\mathrm d\tau f_n\varphi_n(x),
\end{align*}
then we have $u(x,t)=u_1(x,t)+u_2(x,t)$ by the superposition principle. We estimate each term separately. For $u_1$, we have
\begin{align*}
\|u_1\|^2_{L^2(0,T;D(\mathcal A^{\beta/2}))} & =\int_0^T\sum_{n=1}^\infty\left(\lambda_n^{\beta/2}\phi_n E_{\alpha,1}(-\lambda_n^\beta t^\alpha)\right)^2\mathrm d t\\
& \leq C\int_0^T\sum_{n=1}^\infty\phi_n^2\left(\frac{\sqrt{\lambda_n^\beta t^\alpha}}{1+\lambda_n^\beta t^\alpha}\right)^2t^{-\alpha}\,\mathrm d t\\
& \leq C\|\phi\|_{L^2(\Omega)}^2\int_0^T t^{-\alpha}\,\mathrm d t\leq C\|\phi\|_{L^2(\Omega)}^2.
\end{align*}
For $u_2$, in view of Lemmas \ref{lem2.1}--\ref{lem2.13} and Young's inequality for convolutions, we can deduce that
\begin{align*}
\|u_2\|^2_{L^2(0,T;D(\mathcal A^{\beta/2}))} & =\sum_{n=1}^\infty\lambda_n^\beta f_n^2\left\|\int_0^t\mu(\tau)(t-\tau)^{\alpha-1}E_{\alpha,\alpha}(-\lambda_n^\beta(t-\tau)^\alpha)\varphi_n\,\mathrm d\tau\right\|_{L^2(0,T)}^2\\
& \leq\|\mu\|^2_{L^2(0,T)}\sum_{n=1}^\infty\left(\lambda_n^{\beta/2}\int_0^T t^{\alpha-1} E_{\alpha,\alpha}(-\lambda_n^\beta t^\alpha)dt\right)^2 f_n^2\\
& \leq C\|\mu\|^2_{L^2(0,T)}\sum_{n=1}^\infty\left(1-E_{\alpha,1}(-\lambda_n^{\beta/2}T^\alpha)\right)^2 f_n^2\\
& \leq C\|\mu\|^2_{L^2(0,T)}\|f\|^2_{L^2(\Omega)}.
\end{align*}
Therefore, based on the above results, we have $u\in   L^2(0,T;D(\mathcal A^{\beta/2}))$.\medskip 

{\bf Step 2 } We prove $u-\phi\in H_\alpha(0,T;D(\mathcal A^{-\beta/2}))$. Similarly as above, we estimate $u_1-\phi$ and $u_2$ separately. To this end, we first conclude from the results on pp.\! 140--141 in \cite{2006Theory} that
\begin{align*}
\partial_t^\alpha(u(x,t)-\phi(x)) & =D_t^\alpha(u(x,t)-\phi(x))=D_t^\alpha(u_1(x,t)-\phi(x))+D_t^\alpha u_2(x,t)\\
& =-\sum_{n=1}^\infty\phi_n\lambda_n^\beta E_{\alpha,1}(-\lambda_n^\beta t^\alpha)\varphi_n(x)\\
& \quad\,+\sum_{n=1}^\infty f_n\left\{\mu(t)-\lambda_n^\beta\int_0^t\mu(\tau)(t-\tau)^{\alpha-1}E_{\alpha,\alpha}(-\lambda_n^\beta(t-\tau)^\alpha)\,\mathrm d\tau\right\}\varphi_n(x)\\
&=-\mathcal A^\beta u(x,t)+f(x)\mu(t).
\end{align*}
Consequently, from the first step for the spatial regularity estimates, we arrive at
\begin{align*}
\|u-\phi\|_{H_\alpha(0,T;D(\mathcal A^{-\beta/2}))}^2 & =\|\partial_t^\alpha(u-\phi)\|_{L^2(0,T;D(\mathcal A^{-\beta/2}))}^2\\
& =\|-\mathcal A^\beta u+f\,\mu\|_{L^2(0,T;D(\mathcal A^{-\beta/2}))}^2\\
& \leq C\left(\|\phi\|_{L^2(\Omega)}+\|\mu\|_{L^2(0,T)} \|f\|_{L^2(\Omega)}\right)^2.
\end{align*}
Therefore, collecting all the above results, we have $u\in   L^2(0,T;D(\mathcal A^{\beta/2}))$ and $u-\phi\in H_\alpha(0,T;D(\mathcal A^{-\beta/2}))$.
\end{proof}

\section{Proof of Theorem \ref{thm-isp}}\label{sec3}

Before we begin to prove Theorem \ref{thm-isp}, we prepare the following two lemmas.

\begin{lem}[Unique continuation]\label{lem-ucp}
Let ${\phi\in L^2(\Omega)}$, $f=0$ and $u$ be the solution to \eqref{eq-gov}. Then for any nonempty subdomain $\omega\subset\Omega,$ $u=0$ in $\omega\times(0,T)$ implies $u\equiv0$ in $\Omega\times(0,T)$.
\end{lem}

\begin{proof}
The argument resembles that for Sakamoto and Yamamoto \cite[Theorem 4.2]{SY11}, whereas we still provide a proof here for completeness.

Similarly to \cite[Theorem 2.1(i)]{SY11}, one can readily show that the solution $u$ to the initial-boundary value problem \eqref{eq-gov} can be analytically extended from $(0,T)$ to $(0,\infty)$. For simplicity, we still denote this extension by $u$. As was demonstrated in the proof of \cite[Lemma 4.1]{2015Initial-boundary}, the Laplace transform $\widehat u(\,\cdot\,;s)$ of the solution $u(\,\cdot\,,t)$ to \eqref{eq-gov} reads
\[
\widehat u(\,\cdot\,;s)=\sum_{n=1}^\infty\frac{s^{\alpha-1}(\phi,\varphi_n)}{s^\alpha+\lambda^\beta_n}\varphi_n,\quad\mathrm{Re}\,s>s_1,
\]
where $s_1>0$ is a sufficiently large constant. Then it follows from $u=0$ in $\omega\times(0,T)$ that
\[
\sum_{n=1}^\infty\frac{s^{\alpha-1}(\phi,\varphi_n)}{s^\alpha+\lambda^\beta_n}\varphi_n=0\quad\mbox{in }\omega,\ \mathrm{Re}\,s>s_1.
\]
We see that $s^\alpha$ varies over some domain $U\subset\mathbb C$ as $s$ varies over $\mathrm{Re}\,s>s_1$. Therefore, we obtain
\begin{equation}\label{eq3.1}
\sum_{n=1}^\infty\frac{(\phi,\varphi_n)}{\eta+\lambda^\beta_n}\varphi_n=0\quad\mbox{in }\omega,\ \eta=s^\alpha\in U.
\end{equation}
Moreover, we can analytically continue both sides of \eqref{eq3.1}, so that it holds for $\eta\in\mathbb C\setminus\{-\lambda^\beta_n\}_{n=1}^\infty$. Then for any $n=1,2,\dots$, we can take a sufficiently small circle centered at $-\lambda^\beta_n$ excluding distinct eigenvalues and integrate \eqref{eq3.1} on this circle to obtain
\[
u_n:=\sum_{\{k\mid\lambda_k=\lambda_n\}}(\phi,\varphi_k)\varphi_k=0\quad\mbox{in }\omega,\ \forall\,n=1,2,\dots.
\]
Simultaneously, it is readily seen that $u_n$ satisfies the elliptic equation $(\mathcal A-\lambda_n)u_n=0$ in $\Omega$. Then it follows from the unique continuation for elliptic equations that $u_n\equiv0$ in $\Omega$ for all $n=1,2,\dots$. Together with the complete orthogonality of $\{\varphi_n\}_{n=1}^{\infty}$ in $L^2(\Omega)$, this implies $\phi=0$ in $\Omega$ as the initial value of \eqref{eq-gov}. Consequently, the uniqueness for the forward problem guarantees $u\equiv0$ in $\Omega\times(0,T)$, which completes the proof.
\end{proof}

\begin{lem}[Duhamel's principle]\label{lem-dp}
Let $\phi=0,$ $f\in L^2(\Omega)$ and $\mu\in L^2(0,T)$. Then the unique weak solution $u$ to the initial-boundary value problem \eqref{eq-gov} allows the representation
\begin{equation}\label{eq3.2}
u(\,\cdot\,,t)=\int_0^t\theta(t-s)v(\,\cdot\,,s)\,\mathrm d s,\quad0<t<T,
\end{equation}
where $v$ solves the homogeneous problem
\[
\begin{cases}
\partial_t^\alpha v+\mathcal A^\beta v=0 & \mbox{in }\Omega\times(0,T),\\
v=f & \mbox{in }\Omega\times\{0\},\\
v=0 & \mbox{on }\partial\Omega\times(0,T)
\end{cases}
\]
and $\theta\in L^2(0,T)$ stands for the unique solution to the fractional integral equation
\begin{equation}\label{eq3.4}
I^{1-\alpha}\theta(t)=\mu(t),\quad0<t<T.
\end{equation}
\end{lem}

The above conclusion is parallel to \cite[Lemma 4.1]{2016Strong} for the case of a non-fractional elliptic part, hence we omit the proof here.

Now we are well prepared to give the proof of the main theorem.

\begin{proof}[\bf Proof of Theorem \ref{thm-isp}]
Let $u$ satisfy the initial-boundary value problem \eqref{eq-gov} with $\phi=0$, $f\in L^2(\Omega)$ and $\mu\in L^2(0,T)$. Then $u$ takes the form of \eqref{eq3.2} according to Lemma \ref{lem-dp}. Performing the Riemann-Liouville integral operator $I^{1-\alpha}$ on both sides of \eqref{eq3.2}, we deduce
\begin{align*}
I^{1-\alpha}u(\,\cdot\,,t) & =\int_0^t\frac{(t-\tau)^{-\alpha}}{\Gamma(1-\alpha)}\int_0^\tau\theta(\tau-\xi)v(\,\cdot\,,\xi)\,\mathrm d\xi\mathrm d\tau\\
&=\int_0^t v(\,\cdot\,,\xi)\int_\xi^t\frac{(t-\tau)^{-\alpha}}{\Gamma(1-\alpha)}\theta(\tau-\xi)\,\mathrm d\tau\mathrm d\xi\\
& =\int_0^t v(\,\cdot\,,\xi)\int_0^{t-\xi}\frac{(t-\xi-\tau)^{-\alpha}}{\Gamma\left(1-\alpha\right)}\theta(\tau)\,\mathrm d\tau\mathrm d\xi\\
&=\int_0^t v(\,\cdot\,,\xi)I^{1-\alpha}\theta(t-\xi)\,\mathrm d\xi\\
&=\int_0^t\mu(t-\tau)v(\,\cdot\,,\tau)\,\mathrm d\tau,
\end{align*}
where we applied Fubini's theorem and used the relation \eqref{eq3.4}. 

Since the Caputo and the Riemann-Liouville derivatives coincide when the initial value vanishes, we conclude from the governing equation in \eqref{eq-gov} that
\[
D_t^\alpha u=\partial_t^\alpha u=-\mathcal A^\beta u+f\mu=0\quad\mbox{in }\omega\times(T_1,T)
\]
in view of the assumption $u=0$ in $\omega\times(T_1,T)$ and $f=0$ in $\omega$. Together with $u=0$ in $\omega\times(T_1,T)$, it follows from \cite[Theorem 1]{2019An} that $u=0$  in $\omega\times (0,T)$ and therefore
\[
I^{1-\alpha}u(\,\cdot\,,t)=\int_0^t\mu(t-\tau)v(\cdot,\tau)\mathrm d\tau=0\quad\mbox{in }\omega\times(0,T).
\]

According to the Titchmarsh convolution theorem (see \cite[TheoremVII]{T26}), there exist constants $\tau_1,\tau_2\in[0,T]$ satisfying $\tau_1+\tau_2\ge T$ such that
\[
\mu\equiv0\mbox{ in }(0,\tau_1),\quad v\equiv0\mbox{ in }\omega\times(0,\tau_2).
\]
Owing to the key assumption $\mu\not\equiv0$ in $(0,T_0)\subset(0,T)$ in \eqref{eq1.2}, we conclude $\tau_1<T_0$ and thus $\tau_2\ge T-\tau_1>T-T_0\ge0$ or equivalently $v=0$ in $\omega\times(0,\tau_2)$. Then the direct application of Lemma \ref{lem-ucp} immediately indicates $v\equiv0$ in $\Omega\times(0,\tau_2)$ and consequently $f=0$ in $\Omega$ as the initial value of $v$.
\end{proof}

\section{Levenberg-Marquardt method}\label{sec4}

In the sequel, by $u(x,t;f)$ we denote the unique solution to \eqref{eq-gov} in order to emphasize its dependency on $f$. Based on Theorem \ref{thm2.2}, we define a forward operator
\[
\mathcal F:L^2(\Omega\setminus\omega)\ni f\longmapsto u(\,\cdot\,,\,\cdot\,;f)|_{\omega\times(T_1,T)}\in L^2(\omega\times(T_1,T)).
\]

In order to recover the space-dependent source $f$ by the observation data of $u$ in $\omega\times(T_1,T)$, we adopt the Levenberg-Marquardt method as follows. Firstly, we consider the following minimization problem. Take an initial guess $f^0$ and suppose the $k$th 
step approximation $f^k$ has been obtained. Then we should find the $(k+1)$th step approximation $f^{k+1}$ by solving
the minimization problem
\begin{align}
f^{k+1} & =\arg\min\left\{\frac12\left\|\mathcal F(f^k)-u^\delta+\mathcal F'(f- f^k)\right\|_{L^2(\omega\times(T_1,T))}^2\right.\nonumber\\
& \qquad\qquad\quad\:\left.+\frac{\rho_{k+1}}2\left\|f- f^k\right\|_{L^2(\Omega\setminus\omega)}^2\right\},\label{eq5.2}
\end{align}
where $\rho_{k+1}>0$ is regularization parameter at the $(k+1)$th step, and $\mathcal F'$ denotes the Fr\'echet derivative of $\mathcal F$ to $f$ at $f^k$. Besides, $u^\delta$ is the noisy data satisfying
\[
\left\|u(\,\cdot\,,\,\cdot\,;f_{\mathrm{true}})-u^\delta\right\|_{L^2(\omega\times(T_1,T))}\le\delta,
\]
where $f_{\mathrm{true}}\in L^2(\Omega)$ and $\delta>0$ stand for the true solution and the noise level, respectively.

In the following, we use a finite dimensional approximation algorithm to solve the minimization problem \eqref{eq5.2}. Suppose that $\{\chi_n\mid n=1,2,\dots\}$ is an appropriate set of basis functions in $L^2(\Omega\setminus\omega)$. Let
\[
f^{k+1}(x)\approx\sum_{n=1}^N a_n^{k+1}\chi_n(x)\quad\mbox{and}\quad f^k(x)\approx\sum_{n=1}^N a_n^k\chi_n(x),
\]
where $N\in\mathbb N$ and $a_n^{k+1},a_n^k$ ($n=1,2,\dots,N$) are the expansion coefficients. We set
\[
\Phi^N=\operatorname{span}\left\{\chi_1,\chi_2,\dots,\chi_N\right\}
\]
and two vectors $\bm a^{k+1}=(a_1^{k+1},\dots,a_N^{k+1}),\bm a^k=(a_1^k,\dots,a_N^k)\in\mathbb R^N$. We identify approximations $f^{k+1},f^k\in\Phi^N$ with vectors $\bm{a}^{k+1},\bm{a}^k\in\mathbb R^N$, respectively.

Next, we give an inversion algorithm for determining $f$. Based on the above discussions, solving problem \eqref{eq5.2} can be transformed to solving the following minimization problem:
\begin{align}
& \min_{\bm a\in\mathbb R^N}\left\{\frac12\left\|u(\,\cdot\,,\,\cdot\,;\bm a^k)-u^\delta+\nabla_{\bm a^k}u(\,\cdot\,,\,\cdot\,;\bm a^k)(\bm a-\bm a^k)^{\mathrm T}\right\|_{L^2(\omega\times(T_1,T))}^2\right.\nonumber\\
& \qquad\quad\left.+\frac{\rho_{k+1}}2\left\|f-f^k\right\|_{L^2(\Omega\setminus\omega)}^2\right\},\label{eq5.8}
\end{align}
where
\[
\nabla_{\bm a^k}u(\,\cdot\,,\,\cdot\,;\bm a^k)=\left(\frac\partial{\partial a_1^k}u(\,\cdot\,,\,\cdot\,;\bm a^k),\dots,\frac\partial{\partial a_N^k}u(\,\cdot\,,\,\cdot\,;\bm a^k)\right).
\]
We approximate
\[
\frac\partial{\partial a_n^k}u(\,\cdot\,,\,\cdot\,;\bm a^k)\approx\frac{u(\,\cdot\,,\,\cdot\,;a_1^k,\dots, a_n^k+\tau,\dots,a_N^k)-u(\,\cdot\,,\,\cdot\,;a_1^k,\dots,a_n^k,\dots,a_N^k)}\tau
\]
with a suitably chosen numerical differentiation step $\tau$.

The minimizer of \eqref{eq5.8} is simply denoted by $\bm a^{k+1}$ and we write
\begin{equation}\label{eq5.11}
\bm a^{k+1}=\bm a^k+\delta\bm a^k,\quad k=1,2,\dots,
\end{equation}
where $\delta\bm a^k=(\delta a_1^k,\dots,\delta a_N^k)$ is called a perturbation of $\bm a^k$. Thus, in order to get ${\bm a^{k+1}}$, it suffices to compute an optimal perturbation $\delta\bm a^k$. Then the problem \eqref{eq5.8} becomes
\begin{align}
& \min_{\delta\bm a^k\in\mathbb R^N}\left\{\left\|\nabla_{\bm a^k}u(\,\cdot\,,\,\cdot\,;\bm a^k)(\delta\bm a^k)^{\mathrm T}-\left(u^\delta-u(\,\cdot\,,\,\cdot\,;\bm a^k)\right)\right\|_{L^2(\omega\times(T_1,T))}^2\right.\nonumber\\
& \qquad\quad\;\;\;+\delta\bm a^k A^k(\delta\bm a^k)^{\mathrm T}\bigg\},\label{eq5.12}
\end{align}
where $A^k=\operatorname{diag}(\rho_{k+1}((\chi_{i},\chi_{j})_{L^2(\Omega\setminus\omega)})_{N \times N})$.
Let
\[
Q^k=\left(\left(\frac\partial{\partial a_i^k}u(\,\cdot\,,\,\cdot\,;\bm a^k),\frac\partial{\partial a_j^k}u(\,\cdot\,,\,\cdot\,;\bm a^k)\right)_{L^2(\omega\times(T_1,T))}\right)_{N \times N}
\]
and
\[
\bm W^k=\left(\left(u^\delta-u(\,\cdot\,,\,\cdot\,;\bm a^k),\frac\partial{\partial a_i^k}u(\,\cdot\,,\,\cdot\,;\bm a^k)\right)_{L^2(\omega\times(T_1,T))}\right)_{N \times 1} .
\]
We readily verify that minimizing \eqref{eq5.12} is reduced to solving the following normal equation
\[
(Q^k+{A}^k)(\delta\bm a^k)^{\mathrm T}=\bm W^k.
\]
Then by the iteration procedure \eqref{eq5.11}, the optimal approximate solution can be obtained as long as arriving at the given number of iterations.

\section{Numerical experiments}\label{sec5}

In this section, we aim at numerically recovering the unknown spatial component $f(x)$ of the source term in $\Omega\setminus\omega$ from the noisy observation data $u^\delta$ measured in $\omega\times(T_1,T)$. The noisy data is generated by adding a random perturbation to the noiseless data, i.e.,
\[
u^\delta(x,t)=(1+\epsilon\,\mathrm{rand}(-1,1))u(x,t;f_{\mathrm{true}}),\quad(x,t)\in\omega\times(T_1,T),
\]
where $\mathrm{rand}(-1,1)$ denotes a random number distributed uniformly in $(-1,1)$. The corresponding noise level is then calculated by $\delta=\|u^\delta-u(\,\cdot\,,\,\cdot\,;f_{\mathrm{true}})\|_{L^2(\omega\times(T_1,T))}$, where we recall that $\omega$ is a nonempty subdomain of $\Omega$. To evaluate the accuracy of the numerical solution, we compute the relative error defined by
\[
\mathrm{err}=\frac{\|f_K-f_{\mathrm{true}}\|_{L^2(\Omega\setminus\omega)}}{\|f_{\mathrm{true}}\|_{L^2(\Omega\setminus\omega)}}
\]
with the number $K$ of iterations, where $f_K$ is regarded as the reconstructed solution produced by Levenberg-Marquardt method.

The residual $E_k$ at the $k$th iteration is given by
\[
E_k=\left\|u(\,\cdot\,,\,\cdot\,;f^k)-u^\delta\right\|_{L^2(\omega\times(T_1,T))}.
\]
As for the stopping criteria, in this study we use the well-known discrepancy principle, i.e., we choose $K$ satisfying the following inequality
\[
E_K\leq\eta\delta<E_{K-1},
\]
where $\eta>1$ is a constant usually taken heuristically to be 1.01. If the noise level is $0$, then we take $K=40$ for the following examples. The regularization parameter $\rho_k$ at the $k$th step is selected as the sigmoid-type function based on its properties given by
\[
\rho_k=\frac1{1+\exp\left(\gamma_0\left(k-k_0\right)\right)},
\]
where $k_0,\gamma_0>0$ are some tuning parameters chosen in advance.

In the numerical example, we simply set $\Omega=(0,1)$, $\mathcal A=-\partial_x^2$, $T=1$ and select the tuning parameters as $\gamma_0=0.8$ and $k_0=4$. We choose the initial guess and the observation subdomain as $f_0\equiv0$ and $\omega=(0,0.06)$, respectively. For the starting time $T_1$ of the observation, we will change its value to see its influence on the numerical performance. The noiseless data $u(x,t;f_{\mathrm{true}})$ in $\omega\times(T_1,T)$ is obtained by solving the direct problem \eqref{eq-gov} numerically. 

We test the following two choices of true solutions:
\[
f_{\mathrm{true}}^1(x)=x^4+x\sin(\pi x),\quad f_{\mathrm{true}}^2(x)=\mathrm e^{-\pi x^2}+\cos(2\pi x).
\]
Reconstruction results of the source term in $\Omega\setminus\omega$ with various noise levels are shown in Figures \ref{fig1}--\ref{fig2}. We can observe that our algorithm performs worse as $T_1$ increases. This result is definitely consistent with our expectation because the reduction in available information makes the inversion of the algorithm more difficult. We can see from Figures \ref{fig1}--\ref{fig2} and Tables \ref{tab1}--\ref{tab2} that the numerical results match the true solution rather satisfactory up to $1\%$ noise added in the noiseless data for different choices of $\alpha$ and $\beta$. Finally, based on the analysis of numerical results, we verify the validity of the theory and the stability of the algorithm.

\begin{figure}[ht!]\centering
\subfigure[$(\alpha,\beta)=(0.5,0.7)$, $T_1=0.1$.]{\resizebox*{7cm}{4.5cm}{\includegraphics{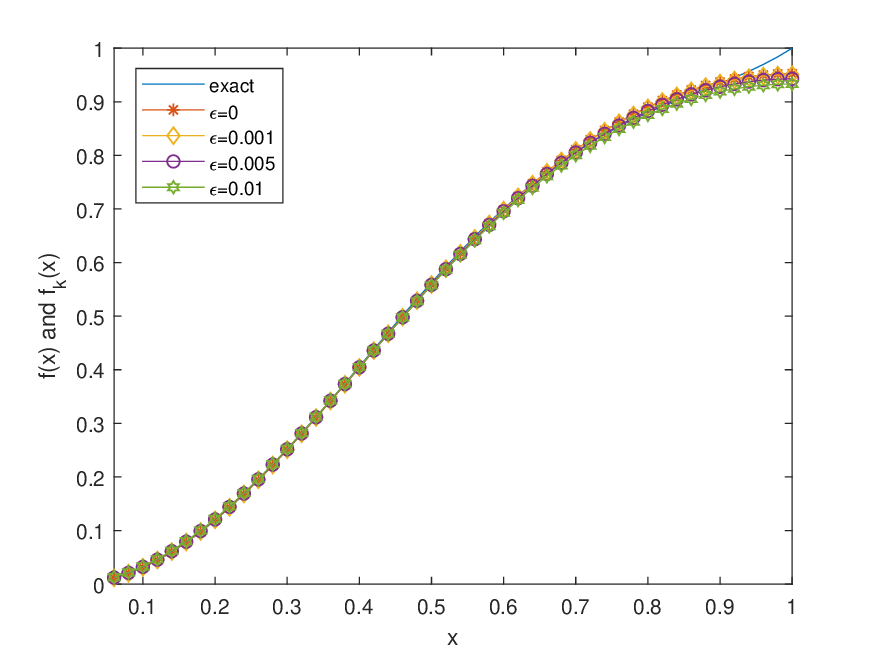}}}
\subfigure[$(\alpha,\beta)=(0.5,0.7)$, $T_1=0.3$.]{\resizebox*{7cm}{4.5cm}{\includegraphics{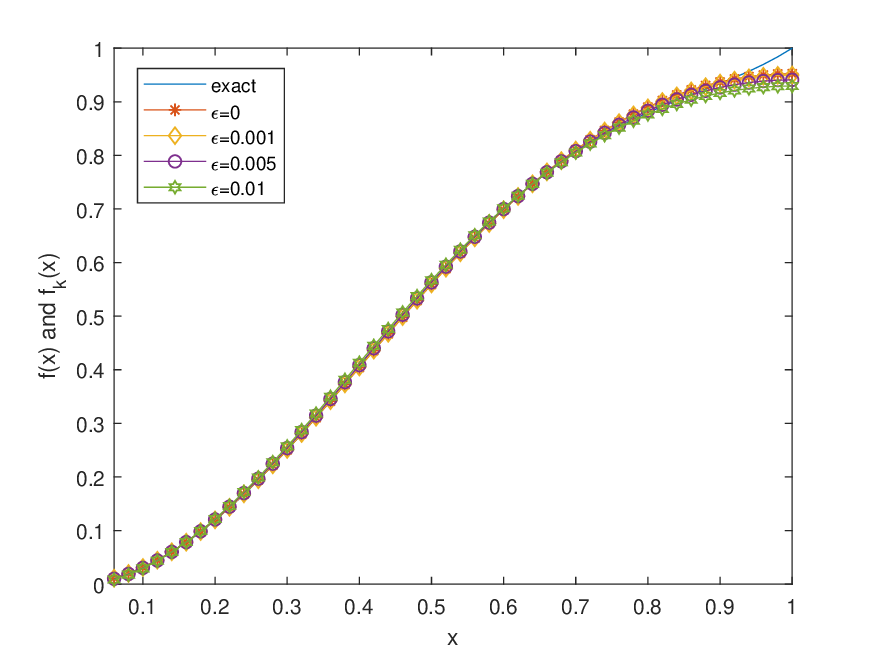}}}
\subfigure[$(\alpha,\beta)=(0.6,0.8)$, $T_1=0.1$.]{\resizebox*{7cm}{4.5cm}{\includegraphics{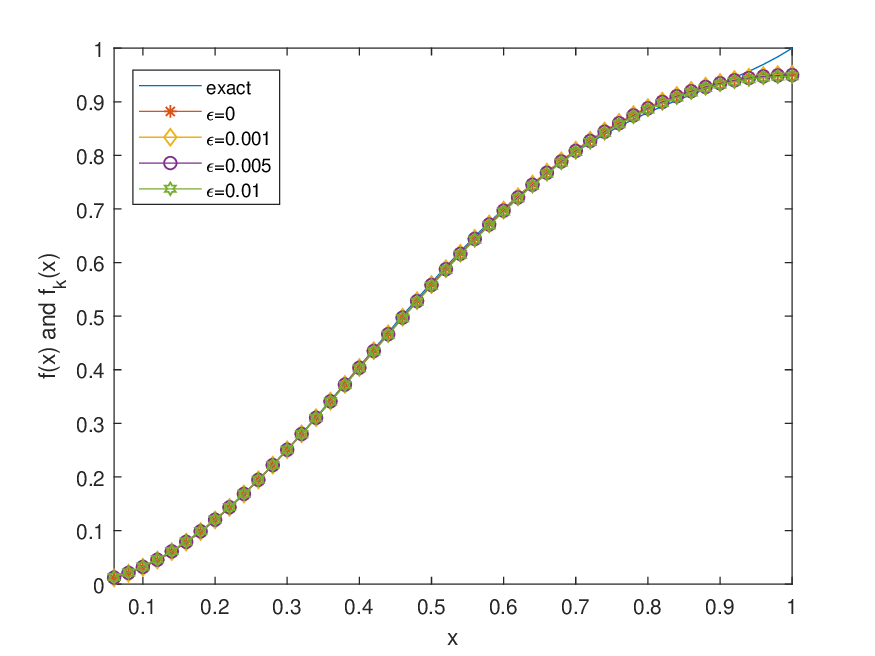}}}
\subfigure[$(\alpha,\beta)=(0.6,0.8)$, $T_1=0.3$.]{\resizebox*{7cm}{4.5cm}{\includegraphics{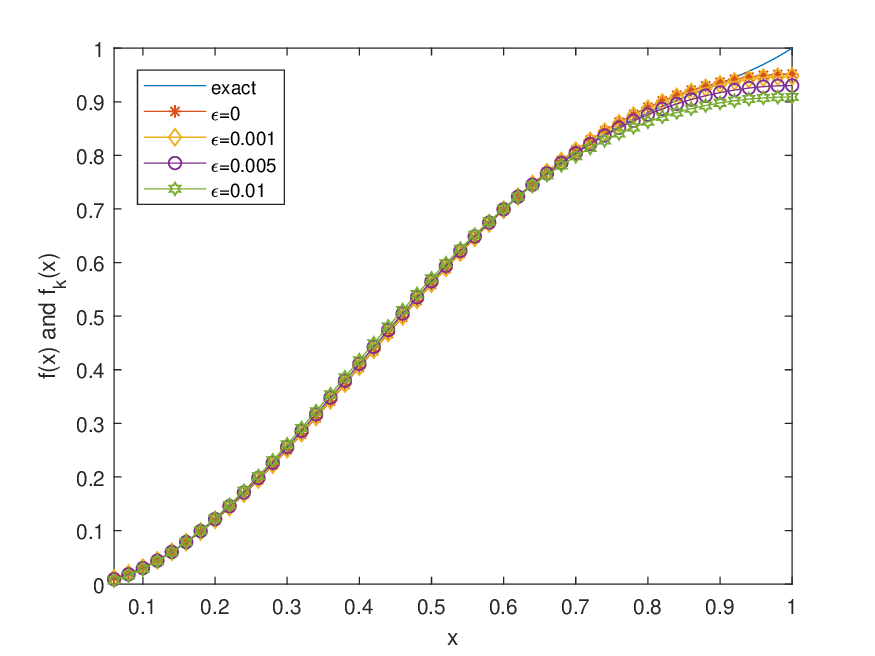}}}
\caption{Reconstruction of the source term $f_{\mathrm{true}}^1$ in $\Omega\setminus\omega$ with various noise levels.}\label{fig1}
\end{figure}	

\begin{figure}[ht!]\centering
\subfigure[$(\alpha,\beta)=(0.5,0.7)$, $T_1=0.1$.]{\resizebox*{7cm}{4.5cm}{\includegraphics{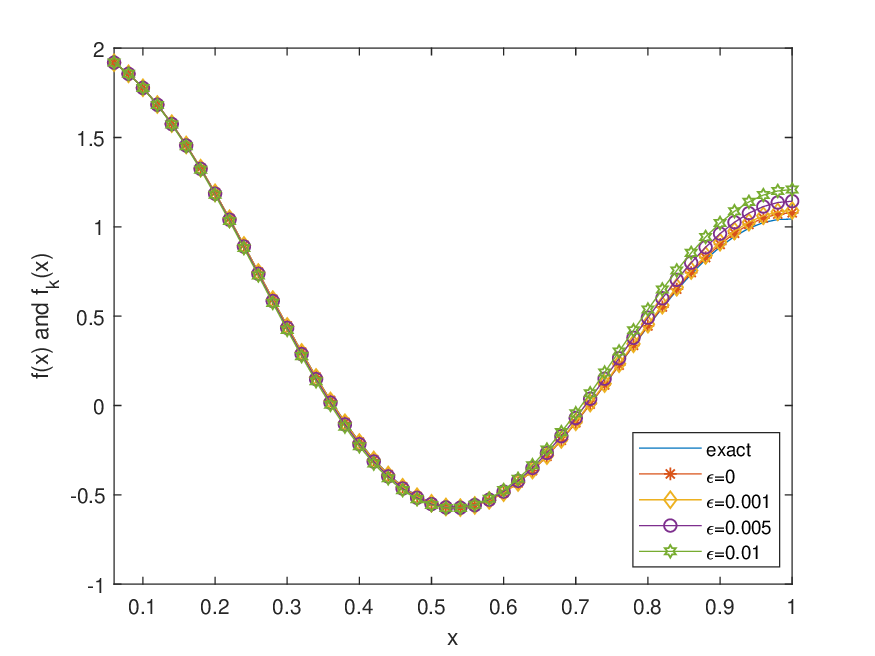}}}
\subfigure[$(\alpha,\beta)=(0.5,0.7)$, $T_1=0.3$.]{\resizebox*{7cm}{4.5cm}{\includegraphics{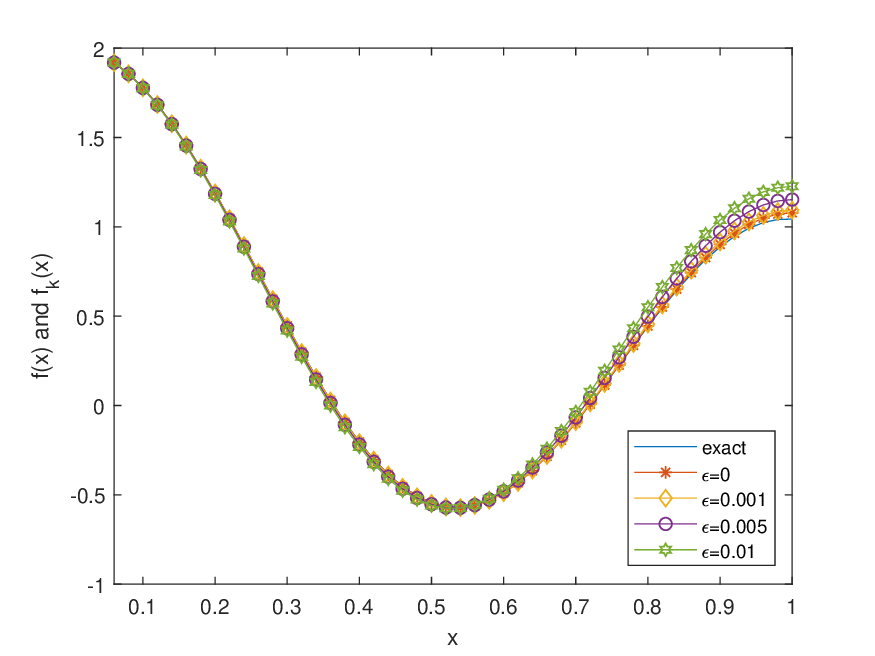}}}
\subfigure[$(\alpha,\beta)=(0.6,0.8)$, $T_1=0.1$.]{\resizebox*{7cm}{4.5cm}{\includegraphics{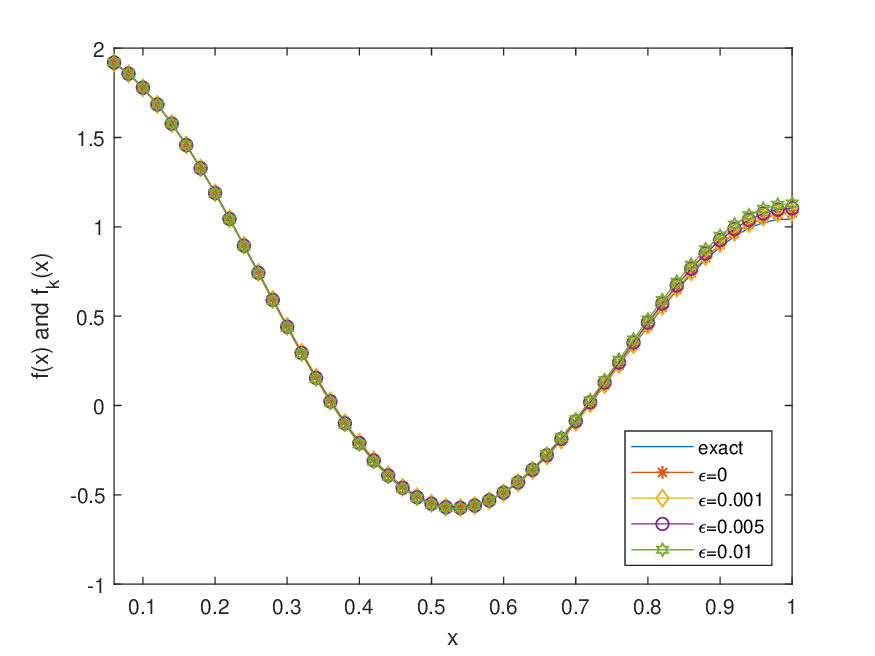}}}
\subfigure[$(\alpha,\beta)=(0.6,0.8)$, $T_1=0.3$.]{\resizebox*{7cm}{4.5cm}{\includegraphics{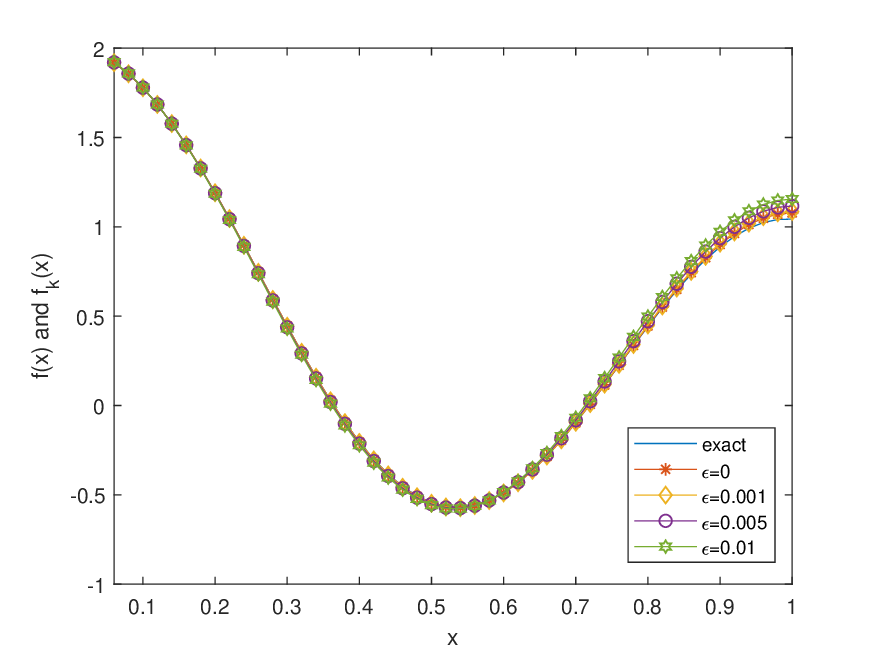}}}
\caption{Reconstruction of the source term $f_{\mathrm{true}}^2$ in $\Omega\setminus\omega$ with various noise levels.}\label{fig2}
\end{figure}

\begin{table}[ht!]\centering
\caption{Relative errors for $f_{\mathrm{true}}^1$ with various noise levels.}\label{tab1}
\small\tabcolsep 0.5pt
\begin{tabular*}{0.7\textwidth}{@{\extracolsep\fill}lcccc}
\toprule 
$\epsilon$ & (a) & (b) &(c) & (d)\\
\midrule
$0$ & $0.0153$ & $0.0153$ & $0.0154$ & $0.0154$\\
$0.001$ & $0.0157$ & $0.0158$ & $0.0155$ & $0.0165$\\
$0.005$ & $0.0184$ & $0.0191$ & $0.0159$ & $0.0247$\\
$0.01$ & $0.0237$ & $0.0254$ & $0.0167$ & $0.0384$\\
\bottomrule
\end{tabular*}
\end{table}

\begin{table}[ht!]
\centering
\caption{Relative errors for $f_{\mathrm{true}}^2$ with various noise levels.}\label{tab2}
\small\tabcolsep 0.5pt
\begin{tabular*}{0.7\textwidth}{@{\extracolsep\fill}lcccc}
\toprule 
$\epsilon$ & (a) & (b) &(c) & (d)\\
\midrule
$0$ & $0.0103$ & $0.0103$ & $0.0102$ & $0.0102$\\
$0.001$ & $0.0163$ & $0.0172$ & $0.0125$ & $0.0136$\\
$0.005$ & $0.0430$ & $0.0477$ & $0.0227$ & $0.0289$\\
$0.01$ & $0.0773$ & $0.0868$ & $0.0362$ & $0.0488$\\
\bottomrule
\end{tabular*}
\end{table}

\section{Conclusion}\label{sec6}

We recover the space-dependent source term from a posteriori interior measurements for the space-time fractional diffusion equations in this paper. The uniqueness of the corresponding inverse problem is established by employing memory effect of fractional derivative, unique continuation property and Duhamel's principle for space-time fractional diffusion equations. We use the Levenberg-Marquardt method to solve the inverse problem. Finally, numerical results verify the correctness of the proposed theory.

\section*{Acknowledgments}

The authors thank the anonymous referees for valuable comments. The second author is supported by National Natural Science Foundation of China (no.\! 12271277) and Ningbo Youth Leading Talent Project (no.\! 2024QL045). The third author is supported by JSPS KAKENHI Grant Numbers JP22K13954, JP23KK0049 and Guangdong Basic and Applied Basic Research Foundation (No.\! 2025A1515012248). This work is partly supported by the Open Research Fund of Key Laboratory of Nonlinear Analysis \& Applications (Central China Normal University), Ministry of Education, China.

\bibliographystyle{unsrt}

\begin{thebibliography}{23}

\bibitem{AF03}
R.A. Adams, Sobolev Spaces, Academic Press, New York, 1975.

\bibitem{2021A fractional}
S. Djennadi, N. Shawagfeh, O.A. Arqub, A fractional Tikhonov regularization method for an inverse backward and source problems in the time-space fractional diffusion equations, {\it Chaos Solitons Fractals}, {\bf150}, 2021, 111127.

\bibitem{2023The}
N.V. Duc, N.V. Thang, N.T. Th\`anh, The quasi-reversibility method for an inverse source problem for time-space fractional parabolic equations, {\it J. Differential Equations}, {\bf344}, 2023, 102--130.

\bibitem{GL15}
R. Gorenflo, Y. Luchko, M. Yamamoto, Time-fractional diffusion equation in the fractional Sobolev spaces, {\it Fract. Calc. Appl. Anal.}, {\bf18}(3), 2015, 799--820.

\bibitem{2020Determination}
J. Janno, Determination of time-dependent sources and parameters of nonlocal diffusion and wave equations from final data, {\it Fract. Calc. Appl. Anal.}, {\bf23}(6), 2020, 1678--1701.	

\bibitem{2023Inverse}
J. Janno, Y. Kian, Inverse source problem with a posteriori boundary measurement for fractional diffusion equations, {\it Math. Methods Appl. Sci.}, {\bf46}(14), 2023, 15868--15882.

\bibitem{2017Harnack's}
J. Jia, J. Peng, J. Yang, Harnack's inequality for a space-time fractional diffusion equation and applications to an inverse source problem, {\it J. Differential Equations}, {\bf262}(8), 2017, 4415--4450.

\bibitem{2022Simultaneous}
Y. Kian, Simultaneous determination of different class of parameters for a diffusion equation from a single measurement, {\it Inverse Problems}, {\bf38}(7), 2022, 075008.

\bibitem{2023Uniqueness}
Y. Kian, Y. Liu, M. Yamamoto, Uniqueness of inverse source problems for general evolution equations, {\it Commun. Contemp. Math.}, {\bf25}(6), 2023, 2250009.

\bibitem{2006Theory}
A.A. Kilbas, H.M. Srivastava, J.J. Trujillo, Theory and Applications of Fractional Differential Equations, Elsevier, Amsterdam, 2006.

\bibitem{2019An}
N. Kinash, J. Janno, An inverse problem for a generalized fractional derivative with an application in reconstruction of time-and space-dependent sources in fractional diffusion and wave equations, {\it Mathematics}, {\bf7}(12), 2019, 1138.

\bibitem{KR20}
A. Kubica, K. Ryszewska, M. Yamamoto, Time-Fractional Differential Equations: A Theoretical Introduction, Springer, Singapore, 2020.

\bibitem{2015Initial-boundary}
Z. Li, Y. Liu, M. Yamamoto, Initial-boundary value problems for multi-term time-fractional diffusion equations with positive constant coefficients, {\it Appl. Math. Comput.}, {\bf257}, 2015, 381--397.

\bibitem{2020Tikhonov}
J. Li, G. Tong, R. Duan, S. Qin, Tikhonov regularization method of an inverse space-dependent source problem for a time-space fractional diffusion equation, {\it J. Appl. Anal. Comput.}, {\bf11}(5), 2021, 2387--2401.

\bibitem{2018An}
Y.S. Li, T. Wei, An inverse time-dependent source problem for a time-space fractional diffusion equation, {\it Appl. Math. Comput.}, {\bf336}, 2018, 257--271.

\bibitem{2016Strong}
Y. Liu, W. Rundell, M. Yamamoto, Strong maximum principle for fractional diffusion equations and an application to an inverse source problem, {\it Fract. Calc. Appl. Anal.}, {\bf19}(4), 2016, 888--906.

\bibitem{2023An inverse}
D. Nie, W. Deng, An inverse random source problem for the time-space fractional diffusion equation driven by fractional Brownian motion, {\it J. Inverse Ill-Posed Probl.}, {\bf31}(5), 2023, 723--738.

\bibitem{1998Fractional}
I. Podlubny, Fractional Differential Equations, Academic Press, San Diego, 1999.

\bibitem{SY11}
K. Sakamoto, M. Yamamoto, Initial value/boundary value problems for fractional diffusion-wave equations and applications to some inverse problems, {\it J. Math. Anal. Appl.}, {\bf382}, 2011, 426--447.

\bibitem{2015Determination}
S. Tatar, R. Tinaztepe, S. Ulusoy, Determination of an unknown source term in a space-time fractional diffusion equation, {\it J. Fract. Calc. Appl.}, {\bf6}(1), 2015, 83--90.

\bibitem{2016Simultaneous}
S. Tatar, R. T\i naztepe, S. Ulusoy, Simultaneous inversion for the exponents of the fractional time and space derivatives in the space-time fractional diffusion equation, {\it Appl. Anal.}, {\bf95}(1), 2016, 1--23.

\bibitem{2015An}
S. Tatar, S. Ulusoy, An inverse source problem for a one-dimensional space-time fractional diffusion equation, {\it Appl. Anal.}, {\bf94}(11), 2015, 2233--2244.

\bibitem{T26}
E.C. Titchmarsh, The zeros of certain integral functions, {\it Proc. London Math. Soc.}, {\bf2}(1), 1926, 283--302

\bibitem{2018Bayesian}
Y. X. Zhang, J. Jia, L. Yan, Bayesian approach to a nonlinear inverse problem for a time-space fractional diffusion equation, {\it Inverse Problems}, {\bf34}(12), 2018, 125002.

\end{thebibliography}

\end{document}